\documentclass{article}
\usepackage[centertags]{amsmath}
\usepackage{hyperref}
\usepackage{amsfonts}
\usepackage{amssymb}
\usepackage{MnSymbol}
\usepackage[bbgreekl]{mathbbol} 
\usepackage{bbm}
\usepackage{amsthm}
\usepackage{newlfont}
\usepackage{amscd}
\usepackage{amsmath,amscd}
\usepackage{verbatim}
\usepackage{enumerate}
\usepackage{accents}
\usepackage{color}
\usepackage[all,2cell]{xy}
\UseAllTwocells
\input xy
\xyoption{2cell}
\xyoption{all}
\usepackage[neverdecrease]{paralist}
\newcommand{\HRightarrow}{\kern0.05ex\vcenter{\hbox{\Huge\ensuremath{\Rightarrow}}}\kern0.05ex}
\newcommand{\hRightarrow}{\kern0.05ex\vcenter{\hbox{\huge\ensuremath{\Rightarrow}}}\kern0.05ex}
\newcommand{\LLRightarrow}{\kern0.05ex\vcenter{\hbox{\LARGE\ensuremath{\Rightarrow}}}\kern0.05ex}
\newcommand{\LRightarrow}{\kern0.05ex\vcenter{\hbox{\Large\ensuremath{\Rightarrow}}}\kern0.05ex}

\newcommand{\HUparrow}{\kern0.05ex\vcenter{\hbox{\Huge\ensuremath{\Uparrow}}}\kern0.05ex}
\newcommand{\hUparrow}{\kern0.05ex\vcenter{\hbox{\huge\ensuremath{\Uparrow}}}\kern0.05ex}
\newcommand{\LLUparrow}{\kern0.05ex\vcenter{\hbox{\LARGE\ensuremath{\Uparrow}}}\kern0.05ex}
\newcommand{\LUparrow}{\kern0.05ex\vcenter{\hbox{\Large\ensuremath{\Uparrow}}}\kern0.05ex}

\newcommand{\HDownarrow}{\kern0.05ex\vcenter{\hbox{\Huge\ensuremath{\Downarrow}}}\kern0.05ex}
\newcommand{\hDownarrow}{\kern0.05ex\vcenter{\hbox{\huge\ensuremath{\Downarrow}}}\kern0.05ex}
\newcommand{\LLDownarrow}{\kern0.05ex\vcenter{\hbox{\LARGE\ensuremath{\Downarrow}}}\kern0.05ex}
\newcommand{\LDownarrow}{\kern0.05ex\vcenter{\hbox{\Large\ensuremath{\Downarrow}}}\kern0.05ex}

\newcommand{\HSearrow}{\kern0.05ex\vcenter{\hbox{\Huge\ensuremath{\Searrow}}}\kern0.05ex}
\newcommand{\hSearrow}{\kern0.05ex\vcenter{\hbox{\huge\ensuremath{\Searrow}}}\kern0.05ex}
\newcommand{\LLSearrow}{\kern0.05ex\vcenter{\hbox{\LARGE\ensuremath{\Searrow}}}\kern0.05ex}
\newcommand{\LSearrow}{\kern0.05ex\vcenter{\hbox{\Large\ensuremath{\Searrow}}}\kern0.05ex}

\newcommand{\HNearrow}{\kern0.05ex\vcenter{\hbox{\Huge\ensuremath{\Nearrow}}}\kern0.05ex}
\newcommand{\hNearrow}{\kern0.05ex\vcenter{\hbox{\huge\ensuremath{\Nearrow}}}\kern0.05ex}
\newcommand{\LLNearrow}{\kern0.05ex\vcenter{\hbox{\LARGE\ensuremath{\Nearrow}}}\kern0.05ex}
\newcommand{\LNearrow}{\kern0.05ex\vcenter{\hbox{\Large\ensuremath{\Nearrow}}}\kern0.05ex}

\newcommand{\HSwarrow}{\kern0.05ex\vcenter{\hbox{\Huge\ensuremath{\Swarrow}}}\kern0.05ex}
\newcommand{\hSwarrow}{\kern0.05ex\vcenter{\hbox{\huge\ensuremath{\Swarrow}}}\kern0.05ex}
\newcommand{\LLSwarrow}{\kern0.05ex\vcenter{\hbox{\LARGE\ensuremath{\Swarrow}}}\kern0.05ex}
\newcommand{\LSwarrow}{\kern0.05ex\vcenter{\hbox{\Large\ensuremath{\Swarrow}}}\kern0.05ex}

\newcommand{\falpha}{\kern0.01ex\vcenter{\hbox{\footnotesize\ensuremath{\alpha}}}\kern0.01ex}
\newcommand{\salpha}{\kern0.05ex\vcenter{\hbox{\small\ensuremath{\alpha}}}\kern0.05ex}
\newcommand{\nalpha}{\kern0.05ex\vcenter{\hbox{\normalsize\ensuremath{\alpha}}}\kern0.05ex}
\newcommand{\lalpha}{\kern0.05ex\vcenter{\hbox{\large\ensuremath{\alpha}}}\kern0.05ex}

\newtheorem{thm}{Theorem}[section]
\newtheorem{cor}[thm]{Corollary}
\newtheorem{lem}[thm]{Lemma}
\newtheorem{prop}[thm]{Proposition}
\theoremstyle{definition}
\newtheorem{defn}[thm]{Definition}

\theoremstyle{remark}
\newtheorem{rem}[thm]{Remark}

\newtheorem{example}{Example}

\DeclareMathOperator{\holim}{holim}

\def\Sim{\mathbb{\Delta}}
\def\sset{s\mathcal{S}et}
\def\csset{{s\mathcal{S}et}^{\mathbb{\Delta}}}
\def\rcsset{{s\mathcal{S}et}^{\mathbb{\Delta}_{r}}}
\newcommand{\tgpd}{\kern0.05ex\vcenter{\hbox{\footnotesize\ensuremath{2}}}\kern0.05ex\mathcal{G}pd} %%\mathpzc{pd}

\def\C{\mathcal{C}}
\def\G{\mathcal{G}}
\def\H{\mathcal{H}}

\def\N{\mathbb{N}}
\def\W{\mathbb{W}}

\def\rar{\rightarrow}

\def\hrar{\hookrightarrow}

\def\oline{\overline}

\def\uline{\underline}

\def\mb{\mathbf}
\def\cl{\mathcal}

\newcommand\ackname{Acknowledgements:}
\if@titlepage
  \newenvironment{acknowledgements}{%
      \titlepage
      \null\vfil
      \@beginparpenalty\@lowpenalty
      \begin{center}%
        \bfseries \ackname
        \@endparpenalty\@M
      \end{center}}%
     {\par\vfil\null\endtitlepage}
\else
  
\fi
\title{Higher Descent Data as a Homotopy Limit}

\author{Matan Prasma}
\date{}
\begin{document}
\maketitle

\begin{abstract}

We define the 2-groupoid of descent data assigned to a cosimplicial 2-groupoid and present it as the homotopy limit of the cosimplicial space gotten after applying the 2-nerve in each cosimplicial degree. This can be applied also to the case of $n$-groupoids thus providing an analogous presentation of ``descent data" in higher dimensions. 

\end{abstract}

\begin{section}{Introduction}
In this note we reinterpret algebro-geometric information, namely descent data, in a homotopically-invariant way. Given a cosimplicial 2-groupoid $\G^\bullet$, its descent data is (the 2-nerve of) a 2-groupoid $\mb{Desc}(\G^\bullet):=Tot_r(\N\G^\bullet)$ (where $Tot_r$ means ``totalization without degeneracies") whose path components coincide with the set of descent data modulo the gauge equivalence relation (see \cite{BGNT} and also \cite[Definitions 1.4, 1.5]{Ye1}). We show that this 2-groupoid is (canonically equivalent to) the homotopy limit $\holim_{\Sim}\N\G^\bullet$ where $\N$ is the 2-nerve applied on each level. Thus, given a weak equivalence of cosimplicial 2-groupoids $\G^\bullet\rar \H^\bullet$, the map $\mb{Desc}(\G^\bullet)\rar \mb{Desc}(\H^\bullet)$ is a weak equivalence of 2-groupoids; this generalizes \cite[Theorem 0.1]{Ye1}. 
We know of two situations in which this setup can arise. 

The first concerns Maurer-Cartan equations. Consider a cosimplicial DGLA, which shows up for instance as the \v{C}ech construction for a sheaf of nilpotent parameter DGLAs. Taking the Deligne 2-groupoid (which encodes solutions to Maurer-Cartan equations) of each cosimplicial degree gives rise to a cosimplicial 2-groupoid. As follows from \cite[Theorem 0.4]{Ye2}, a quasi-isomorphism of cosimplicial pronilpotent DGLAs of quantum type (i.e. concentrated in degrees $\geq -1$) induces a weak equivalence of cosimplicial 2-groupoids.

The second is in the classification of $\G$-gerbes for a sheaf of groups $\G$ (see \cite{Br1}, \cite{Br2}). There, the cosimplicial 2-groupoid arises via the \v{C}ech construction (with respect to a cover) from the sheaf of 2-groups (or crossed modules) $\G \rar Aut(\G)$ and the descent data approximates isomorphism classes of $\G-$gerbes. In some cases, for example when the cover totally trivializes all $\G$-gerbes, $\pi_0\mb{Desc}(\G)$ will classify all $\G-$gerbes and a refinement will yield a weak equivalence of cosimplicial 2-groups.

Descent data is intimately related to non-abelian cohomology. For this reason, the role of codegeneracies is degenerate and we can consider the restricted totalization (see \S\ref{Tot}) which simplifies the homotopical framework. This eliminates the difficulty arising from the fact that the cosimplicial simplicial set gotten by taking the 2-nerve of each level of a cosimplicial 2-groupoid need not be Reedy fibrant (see \cite[Example 9]{Ja}) and gives an argument which is also valid for the case of $n$-groupoids; this is discussed in \S\ref{n-descent}.
\newline
\end{section}

\begin{ackname}
I would like to thank Amnon Yekutieli for introducing and motivating the question at hand and to Yonatan Harpaz for a useful discussion.
\end{ackname}

\begin{section}{Totalization and Restricted Totalization}\label{Tot}
Let $\Sim$ be the category whose objects are non-empty finite ordinals $[0],[1],...,[n]$ where $[n]=\{0,1,...,n\}$ and whose morphisms are weakly order preserving functions. Every morphism in $\Sim$ is a composition of face maps $d^i:[n-1]\rar [n]$ and degeneracies $s^i:[n+1]\rar [n]$, $i=0,...,n$. A simplicial set is a functor $X:\Sim^{op}\rar Set$ and we write $X_n:=X([n])$, $d_i:=X(d^i)$, $s_i:=X(s^i)$. Write $\sset$ for the category whose objects are simplicial sets and whose morphisms are natural transformations. \\

A \emph{cosimplicial object} in a category $\mathcal{C}$ is a functor $\Sim\rar \cl{C}$. In particular, a \emph{cosimplicial simplicial set} is a cosimplicial object in $\sset$. We write $\csset$ for the category whose objects are cosimplicial simplicial sets and whose morphisms are natural transformations.
If $X$ is a cosimplicial object we will denote the object assigned to $[n]$ by $X^n$. The maps $d^i:=X(d^i)$ and $s^i:=X(s^i)$ are called cofaces and codegeneracies respectively. The \emph{cosimplicial standard simplex} $\Delta$ has $\Delta^n$ in its $n$-th cosimplicial degree and cofaces and codegeneracies induced by precomposition. 
For $X,Y\in\csset$, the product $X\times Y$ is the cosimplicial simplicial set with $(X\times Y)^n:=X^n\times Y^n$ and for $A\in\sset$, we write, by abuse of notation, $A$ for the constant cosimplicial simplicial set with $A^n:=A$ for all $n$ and cofaces and codegeneracies being identities.  

The category $\csset$ is enriched over simplicial sets. Given $X,Y\in \csset$, the `internal hom' $\uline{\csset}(X,Y)$ is the simplicial set whose $n$-simplices are 
$$\uline{\csset}(X, Y)_n=\csset(X\times \Delta^n,Y)$$
Here, $X\times \Delta^n$ is the product of $X$ with the constant cosimplicial simplicial set $\Delta^n$.

With this enrichment, $\csset$ is a simplicial category in the sense of \cite[II,2.1]{GJ} or in our terminology, \emph{tensored and cotensored} over $\sset$ (see \cite[II, 2.5]{GJ}). For $A,B\in\csset$ we denote the tensor and cotensor
functors by $$A\times (-): \sset\rar \csset\;\;and\;\;B^{(-)}:(\sset)^{op} \rar \csset$$ respectively; these are the left adjoints of $\uline{\csset}(A,-)$ and $\uline{\csset}(-,B)$.
\begin{defn}
The \emph{totalization} $Tot:\csset\rar \sset$ is the simplicial set $Tot(X^\bullet)=\uline{\csset}(\Delta^\bullet,X^\bullet)$. 
\end{defn}

We let $\Sim_{r}$ denote the subcategory of $\Sim$ with the same objects but only injective maps i.e. compositions of face maps $d^i$. A \emph{restricted cosimplicial object} in a category $\cl{C}$ is a functor $\Sim_{r}\rar \cl{C}$; it is also called a \emph{semi-cosimplicial object} by some authors. In particular, a restricted cosimplicial object in $\sset$ is called a \emph{restricted cosimplicial simplicial set}. There is an obvious `restriction' functor $r:\csset\rar \rcsset$ and in particular we have $r\Delta\in\rcsset$. The category $\rcsset$ is again enriched over simplicial sets so that if $X,Y\in \rcsset$ we denote $\uline{\rcsset}(X,Y)\in \sset$. Its $n$-simplices are $\uline{\rcsset}(X,Y)_n:=\rcsset(X\times r\Delta^n,Y)$ and given $\theta:[m]\rar [n]$ in $\Sim$, the map $\theta^*:\uline{\rcsset}(X,Y)_n\rar \uline{\rcsset}(X,Y)_m$ is induced by composing with the map $\theta_*:r\Delta^m\rar r\Delta^n$. Simplicial identities hold since their opposites hold in $\Delta$. The arguments in \cite[II,2.5]{GJ} may be used verbatim to show that $\rcsset$ is tensored and cotensored over $\sset$.
\begin{defn}
The \emph{restricted totalization} is the functor $Tot_r:\rcsset\rar \sset$ defined by $Tot_r(X^\bullet)=\uline{\rcsset}(r\Delta^\bullet,X^\bullet)$.
\end{defn}
More generally we can use ends (see \cite[IX.5]{Ma}) to get:
%(i.e. satisfies conditions (1) and (3) of \cite[II,2.1]{GJ}) 
\begin{defn}\label{abstract tot}
Let $\C$ be a category cotensored over simplicial sets.
\begin{enumerate}
\item The \emph{totalization} of $\G^\bullet \in \C^{\Sim}$ is the object of $\C$ is given by the end  $$Tot(\G^\bullet):=\int_{[n]\in \Sim}(\G^n)^{\Delta^n}.$$
\item The \emph{restricted totalization} of $\G^\bullet \in \C^{\Sim_r}$ is the object of $\C$ is given by the end $$Tot_r(\G^\bullet):=\int_{[n]\in \Sim_r}(G^n)^{\Delta^n}.$$
\end{enumerate}
\end{defn}

\end{section}

\begin{section}{Model structures}
We assume the reader is familiar with the definition of a model category. Let us shortly spell out the definition of a simplicial model category.

\begin{defn}
A model category $\cl{M}$ is called \emph{simplicial} if it is enriched with tensor and cotensor over $\sset$ and satisfies the following axiom \cite[II.2 SM7]{Qu}:
If $f:A\rar B$ is a cofibration in $\cl{M}$ and $i:K\rar L$ is a cofibration in $\sset$ then the map 
$$q:A\otimes L\coprod_{A\otimes K} B\otimes K\rar B\otimes L$$
\begin{enumerate}
\item is a cofibration;
\item is a weak equivalence if either
\begin{enumerate}
\item $f$ is a weak equivalence in $\cl{M}$ or
\item $i$ is a weak equivalence in $\sset$.
\end{enumerate}
\end{enumerate}
\end{defn}

\begin{defn}
A category $\cl{R}$ is called a \emph{Reedy category} if it has two subcategories $\cl{R}_+,\cl{R}_-\subseteq \cl{R}$ and a \emph{degree} function $d:ob(\cl{R})\rar \alpha$ where $\alpha$ is an ordinal number such that:
\begin{itemize}
\item Every non-identity morphism in $\cl{R}_+$ raises degree;
\item Every non-identity morphism in $\cl{R}_-$ lowers degree;
\item Every morphism in $\cl{R}$ factors uniquely as a map in $\cl{R}_-$ followed by a map in $\cl{R}_+$.
\end{itemize}
\end{defn}

The category $\Sim$ is a Reedy category with $\Sim_+=\Sim_{inj}\;\;(=\Sim_{r})$, $\Sim_-=\Sim_{surj}$ and the obvious degree function.

Let $\cl{R}$ be a Reedy category and $\cl{C}$ any category. Given a functor $X:\cl{R}\rar \cl{C}$ and an object $n\in\cl{R}$ we set $X^n:=X(n)$ (to relate to the case $\cl{R}=\Delta$), and define the $n$-th \emph{latching object} to be

$$L^nX=colim_{\mathbb{L}(\cl{R})}X^s$$
where $\mathbb{L}(\cl{R})$ is the full subcategory of the over category  ${\cl{R}}_+/n$ containing all objects except the identity $id_n$. 

Dually, define the $n$-th \emph{matching object} to be

$$M^nX=lim_{\mathbb{M}(\cl{R})}X^s$$
where $\mathbb{M}(\cl{R})$ is the full subcategory of the under category $n/\cl{R}_-$ containing all objects except $id_n$. We have natural morphisms
$$L^nX\rar X^n\rar M^nX.$$
The importance of a Reedy structure on $\cl{R}$ is due to the following:
\begin{thm}\cite{Re}
Let $\cl{R}$ be a Reedy category and $\cl{M}$ a model category. The functor category $\cl{M}^{\cl{R}}$ admits a structure of a model category, called \emph{Reedy model structure} in which a map $X\rar Y$ is a
\begin{itemize}
\item  Weak equivalence iff $X^n\rar Y^n$ is a weak equivalence in $\cl{M}$ \\ for every $n$.
\item  Cofibration iff the map $L^nY\coprod_{L^nX}X^n\rar Y^n$ is a cofibration in $\cl{M}$\\ for every $n$.
\item  Fibration iff the map $X^n\rar M^nX\times_{M^nY}Y^n$ is a fibration in $\cl{M}$\\ for every $n$.
\end{itemize}
In particular, an object $X$ is
\begin{itemize}
\item Fibrant iff $X^n\rar M^n X$ is a fibration in $\cl{M}$ for every $n$.
\item Cofibrant iff $L^nX\rar X^n$ is a cofibration in $\cl{M}$ for every $n$.
\end{itemize}
Moreover \cite[Theorem 4.7]{An}, if the model structure on $\cl{M}$ is simplicial, so is the Reedy model structure on $\cl{M}^\cl{R}$.
\end{thm}
\begin{cor}
The Kan-Quillen model structure $\sset_{K-Q}$ and the Reedy structure on $\Sim$ (respectively $\Sim_{r}$) induce a simplicial model structure on $\csset$ (respectively $\rcsset$).
\end{cor}
\begin{example}\label{simplex}
The object $X=\Delta^\bullet\in \csset$ is Reedy cofibrant. The map $L^nX\rar X^n$ is the inclusion $\partial \Delta^n \hrar \Delta^n$ which is a cofibration of simplicial sets.
\end{example}
Next, we recall another model structure on $\rcsset$.
\begin{thm}
The simplicial enrichment of $\rcsset$ can be extended to a simplicial model structure, called the \emph{projective model structure}, in which a map $X\rar Y$ is a
\begin{itemize}
\item
weak equivalence if for each $n$, $X^n\rar Y^n$ is a weak equivalence.
\item
fibration if for each $n$, $X^n\rar Y^n$ is a Kan fibration. 
\item
cofibration if it has the left lifting property with respect to trivial fibrations. 
 \end{itemize}
In particular, $X$ is a fibrant object iff $X^n$ is a Kan complex for every $n$. 
\end{thm}
\end{section}
Suppose $\cl{R}$ is a Reedy category and $\cl{M}$ is a model category. In general, if the projective model structure on $\cl{M}^\cl{R}$ exists (e.g. when $\cl{M}$ is sufficiently nice) it will be very different than the Reedy model structure. However, in special cases the two may coincide. 
\begin{prop}%\cite[A.2.9.22]{Lur}
If $\cl{R}=\cl{R}_+$ the projective and Reedy model structures on $\cl{M}^\cl{R}$ coincide.
\end{prop}
\begin{proof}
In this case, for every $X\in \cl{M}$ the $n$-th matching object $M^nX$ equal the terminal object, being the limit over the empty diagram, so that a map $X\rar Y$ is a Reedy fibration iff $X^n\rar Y^n$ is a fibration in $\cl{M}$. This means that the two model structures have the same classes of weak equivalences and fibrations, and hence coincide.    
\end{proof}

For $\cl{R}=\Sim_{r}$ we obtain:
\begin{cor}\label{restricted simplex}
The Reedy and projective model structures on $\rcsset$ coincide.
Thus, an object $X\in \rcsset$ is Reedy fibrant iff $X^n$ is a Kan complex for each $n$.
\end{cor}

\begin{rem}\label{rDelta is cofibrant}
By example \ref{simplex}, $\Delta^\bullet$ is Reedy cofibrant in $\csset$ and since the indexing category defining $L^n\Delta^\bullet$ depends only on $\Sim_+=\Sim_{r}$, we have $L^n\Delta^\bullet =L^n r\Delta^\bullet$. Thus, the map $L^n r\Delta^n\rar r\Delta^n$ is again the inclusion $\partial \Delta^n \hrar \Delta^n$ so that $r\Delta^\bullet$ is Reedy cofibrant in $\rcsset$.
\end{rem}

\begin{section}{2-Groupoids}\label{2-groupoids}

\begin{defn}
A (strict) \emph{2-groupoid} is a groupoid-enriched (small) category in which all morphisms are invertible.

Explicitly, a 2-groupoid consists of:
\begin{itemize}
\item
a set of \emph{objects};
\item
for every pair of objects $x,y$, a set of \emph{1-morphisms}, written as $f:x\rar y$; and, for every object $x$, a distinguished 1-morphism $1_x:x\rightarrow x$; 
\item
for every pair of 1-morphisms $f,g:x\rightarrow y$ a set of \emph{2-morphisms}, written as $a:f\Rightarrow g$; and, for every 1-morphism $f$, a distinguished 2-morphism $1_f:f\Rightarrow f$
\end{itemize} 
together with a composition law for 1-morphisms and vertical and horizontal composition laws for 2-morphisms (denoted by $*$ and $\circ$ respectively) subject to three axioms, expressing associativity of composition and left and right unit laws and in addition satisfy the `interchange law': $$(b *a)\circ (b' *a')=(b'\circ b)*(a'\circ a).$$ All morphisms are invertible with respect to these composition laws.
\end{defn}

There are 2-categorical analogues for the notions of a functor and natural transformation. However, since 2-categories have 2-morphisms, an additional `level of arrows' reveals itself, namely, the one of modifications. There is some ambiguity regarding these notions, since one can consider also their weak versions. For the sake of clarity, we spell out the definitions we use, which are taken from \cite[I,2.2;I,2.3]{Gr}.
\begin{defn}\label{cartesian closed structure}
Let $\G,\H$ be a pair of 2-groupoids. 
\begin{enumerate}[(I)]
\item
A (strict) \emph{2-functor} $\Phi:\G\rar \H$ is a groupoid-enriched functor between the underlying groupoids of $\G$ and $\H$. Explicitly, $\Phi$ assigns:
\begin{itemize}
\item to each object $x\in \G$, an object $\Phi x\in \H$,
\item to each 1-morphism $f:x\rar y\in \G$, a 1-morphism $\Phi f:\Phi x\rar \Phi y\in \H$,
\item to each 2-morphism $a:f\Rightarrow g\in \G$ a 2-morphism $\Phi a:\Phi f\Rightarrow \Phi g\in \H$
\end{itemize}
and this assignment respects all compositions and units.  

\item
Given a pair of 2-functors $\Phi,\Psi:\G\rar \H$ between 2-groupoids, a (strict) \emph{2-natural transformation} $\Theta:\Phi\Rightarrow \Psi$ consists of a 1-morphism $\eta_x:\Phi x\Rightarrow \Psi x$ for every object $x\in \G$ which is natural in the sense that for every 2-morphism $a:f\Rightarrow g$ in $\G$, the diagram 
$$\xymatrix{\Phi(x)\ar[d]_{\eta_x}\rtwocell_{\Phi f}^{\Phi g}{\Phi a} & \Phi(y)\ar[d]^{\eta_y} \\ \Psi(x)\rtwocell^{\Psi f}_{\Psi g}{\Psi a} & \Psi(y)}$$ is commutative in that $1_{\eta_y}\circ \Phi a= \Psi a \circ \eta_x$ as 2-morphisms in $\H$.

\item Given a pair of 2-natural transformations $\eta,\theta:\Phi\Rightarrow \Psi$, a (strict) \emph{modification} $\mu:\eta \Rrightarrow \theta$ consists of a 2-morphism $\mu_x:\eta_x\Rightarrow \theta_x$ in $\H$ for every object $x\in \G$ such that for every 1-morphism $f:x\rar y$ in $\G$, the diagram 
$$\xymatrix{\Phi(x)\ar[d]_{\Phi f}\rtwocell_{\eta_x}^{\theta_x}{\mu_x} & \Psi(x)\ar[d]^{\Psi f} \\ \Phi(y)\rtwocell^{\eta_y}_{\theta_y}{\mu y} & \Psi(y)}$$ is commutative in the sense of (II).
\end{enumerate}
\end{defn}

We denote by $\tgpd$ the category of 2-groupoids and strict 2-functors between them. The collection of 2-functors from $\G$ to $\H$, their 2-natural transformations and their modifications is naturally a 2-category (see \cite[2.3]{Gr}) which is in fact a 2-groupoid because of invertibility of 1-and 2-morphisms in the codomain $\H$. We denote this $2$-groupoid by $\uline{\tgpd}(\G,\H)$. 
\begin{thm}\label{cartesian closeness}{(cf. \cite[2.3]{Gr})}
The category $\tgpd$ is cartesian closed with respect to $\uline{\tgpd}(\G,\H)$. 
\end{thm}

Let $\Sim_{\leq n}$ be the full subcategory of $\Sim$ with objects $[0],...,[n]$ and let $\sset_{\leq n}$ be the category of functors $(\Sim_{\leq n})^{op}\rar \mathcal{S}et$. Objects of $\sset_{\leq n}$ are called \emph{$n$-truncated simplicial sets}. The inclusion $\Sim_{\leq n}\rar \Sim$ induces a `truncation functor' $tr_n:\sset\rar \sset_{\leq n}$ which admits right and left adjoints $cosk_n:\sset_{\leq n}\rar \sset$ and $sk_n:\sset_{\leq n}\rar \sset$ respectively. We denote by $Cosk_n:\sset\rar \sset$ the composition $cosk_n\circ tr_n$ and by $Sk_n$ the composition $sk_n\circ tr_n$. The functor $Sk_n$ takes a simplicial set and creates a new simplicial set from its $n$-truncation by adding degenerate simplices in all levels above $n$; it is the simplicial analogue of the $n$-skeleton of a $CW$ complex. The functor $Cosk_n$ has a more involved simplicial description; it is the simplicial analogue of the $(n-1)$th Postnikov piece $P_{n-1}$. \\
By abstract considerations, one can show that $Cosk_n$ is right adjoint to $Sk_n$. 
Thus, a map $X\rar Cosk_n Y$ correspond precisely to a map $Sk_n X\rar Y$. 
\begin{defn}
A simplicial set $X$ is called \emph{$n$-coskeletal} if the canonical map $X\rar Cosk_n X$ is an isomorphism. 
\end{defn}

In particular, given an $n$-truncated simplicial set $X$, $cosk_n X$ is an \linebreak$n$-coskeletal simplicial set. Thus, in order to define an $n$-coskeletal simplicial set it is enough to define its $n$-truncation. 

In \S\ref{vertices and homotopies} we intend to interpret the definition of descent data in terms of the 2-nerve. In order to improve readability, we now rewrite the definition of \cite[\S2]{MS} with the notations relevant for our formulae. 

\begin{defn}\label{2-nerve def}
The \emph{2-nerve} is the functor $\N:\tgpd\rar \sset$ which takes a 2-groupoid $\G$ to the 3-coskeletal simplicial set $\N\G$ whose
\begin{itemize}
\item 0-simplices are the objects of $\G$;
\item 1-simplices are the morphisms of $\G$;
\item 2-simplices are triangles of the form 
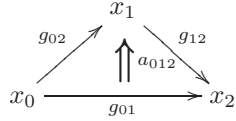
\[
\xymatrix @=0.65pc{& & x_1\ar[ddrr]^{g_{12}} & &\\& & & & \\ x_0\ar@{}[uu]_(0.4){\;\;\;\;\;\;\;\;\;\;\;\;\;\;\;\;\LLUparrow{a}_{012}}\ar[uurr]^{g_{02}}\ar[rrrr]_{g_{01}} & & & & x_2 }
\]

where $g_{ij}:x_i\rightarrow x_j$ and $\alpha:g_{02}\Rightarrow g_{12}\circ g_{01}$ are 1-and 2-morphisms (respectively) in $\G$;
\item 3-simplices are commutative tetrahedra of the form

$
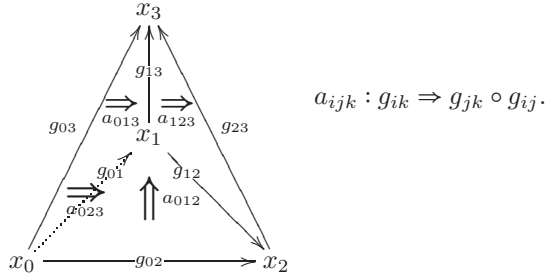
\xymatrix@!=2pc{ & x_3 & \\ & x_1\ar@{}[u]^(0.2){\;\;\;\;\overset{\LRightarrow}{a_{013}}}_(0.2){\overset{\LRightarrow}{a_{123}}}\ar[u]|{g_{13}}\ar[dr]|(0.3){g_{12}} & \\ x_0\ar@{}[ur]|{\overset{\LLRightarrow}{a_{023}}}\ar[uur]^{g_{03}}\ar@{.>}[ur]|(0.7){g_{01}}\ar[rr]|{g_{02}}\ar@{}[u]_(0.5){\;\;\;\;\;\;\;\;\;\;\;\;\;\;\;\;\;\;\;\;\;\LLUparrow a_{012}} & & x_2\ar[uul]_{g_{23}}}$  $\xymatrix{\\a_{ijk}:g_{ik}\Rightarrow g_{jk}\circ g_{ij}.}$

Commutativity of this tetrahedron means that the diagram of 2-morphisms

\begin{equation}\label{tetrahedron commutativity}
\xymatrix@=1.5pc{g_{03}\ar[rr]^{a_{023}}\ar[d]_{a_{013}}  && g_{23}\circ g_{02}\ar[d]^{1_{g_{23}}\circ a_{012}} \\ g_{13}\circ g_{01}\ar[rr]_{a_{123}\circ 1_{g_{01}}} && g_{23}\circ g_{12}\circ g_{01}}
\end{equation}
commutes.

\end{itemize}
\end{defn}

We will need four well-known properties of the 2-nerve:
\begin{prop}\cite{MS}\label{2-nerve}
\begin{enumerate}
\item $\N$ preserves products.
\item For every 2-groupoid $\G$, $\N\G$ is a Kan complex.
\item A map of 2-groupoids $\G\rar \H$ is a weak equivalence iff $\N\G\rar \N\H$ is a weak equivalence of simplicial sets.
\item $\N$ admits a left adjoint $\W:\sset\rar \tgpd$, called the \emph{Whitehead 2-groupoid}. \label{W}
\end{enumerate}
\end{prop}

The category $\tgpd$ admits a natural simplicial enrichment via $\N\uline{\tgpd}(-,-)$. This enrichment is nicely behaved in the following sense:
\begin{prop}\label{2gpd is t-c}
The simplicially-enriched category $\tgpd$ is tensored and cotensored over $\sset$.
\end{prop}
\begin{proof}
We need to verify the conditions of \cite[II,2.1]{GJ}. The functor $((-)\times \G)\circ \W$ is a left adjoint to $\N\tgpd(\G,-)$ and the functor $\uline{\tgpd}(W(-),\H)$ is a left adjoint to $\N\uline{\tgpd}(-,\H)$. 
\end{proof}

\begin{rem}
It is worth notice that \cite[5.1]{No} shows the inner hom described in \ref{cartesian closed structure} does not induce a simplicial model category structure on $\tgpd$ via setting the simplicial mapping space to be $\N\uline{\tgpd}(\G,\H)$. However in the current note, the main homotopical part is done in the category of (cosimplicial) simplicial sets so that we do not need a full-fledged homotopy theory of $\tgpd$.
\end{rem}

We shall need a slight generalization of proposition \ref{2gpd is t-c}:
\begin{prop}%(cf. Hirschhorn Theorem 11.7.3.)
If $\C$ is tensored and cotensored over $\sset$ and $I$ is any small category, then the functor category $\C^I$ is again tensored and cotensored over $\sset$.
\end{prop}
\begin{proof}
Denote the inner homs and their adjoints by 

$\xymatrix{\C\ar@<-0.5ex>[r]_{\uline{\C}(X,-)} & \sset\ar@<-0.5ex>[l]_{X\otimes(-)}}$ and $\xymatrix{\C^{op}\ar@<-0.5ex>[r]_{\uline{\C}(-,Y)} & \sset\ar@<-0.5ex>[l]_{Y^{(-)}}}$. 

One defines for $\undertilde{X}\in \C^I$ and $K\in \sset$, $(\undertilde{X}\otimes K)_\alpha :=\undertilde{X}_\alpha\otimes K$ and $(\undertilde{X}^K)_\alpha:=(\undertilde{X}_\alpha)^K$ for every $\alpha\in I$. Then, $\uline{\C^I}(\undertilde{X},\undertilde{Y})_n=\C^I(\undertilde{X}\otimes\Delta^n, \undertilde{Y})$ with the obvious face and degeneracy maps provides the desired inner hom.  
\end{proof}
\begin{cor}\label{cosimplicial 2-groupoids are t-c}
The categories $\tgpd^\Sim$ and $\tgpd^{\Sim_r}$ are tensored and cotensored over simplicial sets.
\end{cor}
The last corollary enables us to express the totalization as an end via Definition \ref{abstract tot}.

By abuse of notations, we denote by $\N,\;\W$ the prolongation of the 2-nerve and Whitehead 2-groupoid functors to the categories $\tgpd^{\Sim},\;\csset$ (respectively). Since (level-wise) coproducts define the tensoring (over $\sset$) in $\tgpd$ and $\csset$ and $\W$ commutes with coproducts, the premisses of \cite[Lemma 2.9(1)]{GJ} are satisfied and we have:
\begin{prop}\label{Enriched adjunction}
There is an enriched adjunction
$$\uline{\tgpd^{\Sim}}(\W \G^\bullet,\H^\bullet) \cong \uline{\csset}(\G^\bullet,\N\H^\bullet)$$
\end{prop}  

\end{section}

\begin{section}{Descent data of cosimplicial 2-groupoids}\label{vertices and homotopies}
Following \cite{BGNT}, descent data of a cosimplicial crossed groupoid is defined in \cite{Ye1}. Since crossed groupoids can be viewed precisely as 2-groupoids (e.g. as a special case of \cite{BH1}), a translation leads to the following:

\begin{defn}(cf. \cite[Definition 1.4]{Ye1})
Given a cosimplicial 2-groupoid $\G^\bullet=\{\G^n\}$, a \emph{descent datum} is a triple $(x,g,a)$ in which:
\begin{enumerate}
\item $x$ is an object of $\G^0$;
\item $g:d^1x\rightarrow d^0x$ is a 1-morphism in  $\G^1$ and
\item $a:d^1g\Rightarrow d^0g\circ d^2g$ is a 2-morphism in $\G^2$.
\end{enumerate} 
such that 
\begin{equation}\label{tetrahedron}
\tag{twisted 2-cocycle}
(1_{d^1d^0g}\circ d^3 a)*d^1a=(d^0a \circ 1_{d^2d^2g})*d^2a. 
\end{equation}

\end{defn}
Let $(x,g,a)$ be a descent datum of $\G^\bullet$.
\newline
\newline 
Write $x_i\equiv x_i^{(1)}\; (i=0,1)$ for the object of $G^1$ corresponding to the vertex $(i)$ of $\Delta^1$, i.e. $x_i=d^jx$ where $\{j\}=\{0,1\}\setminus \{i\}$; thus $g:x_0\rightarrow x_1$.
\newline
\newline
Similarly, write $x_i\equiv x_i^{(2)}\; (i=0,1,2)$ and $g_{ij}\equiv g_{ij}^{(2)} \; (0\leq i<j\leq 2)$ for (respectively) the object and 1-morphism of $\G^2$ corresponding to the vertex $(i)$ and edge $(ij)$ of $\Delta^2$. In other words, $x_i =d^kd^jx$ where $\{j<k\}=\{0,1,2\}\setminus \{i\}$ and $g_{ij}=d^kg$ where $\{k\}=\{0,1,2\}\setminus \{i,j\}$; thus $g_{ij}:x_i\rightarrow x_j$ and $a:g_{02}\Rightarrow g_{12}\circ g_{01}$.
\newline
\newline
Finally, write $x_i \equiv x_i^{(3)}\; (i=0,...,3)$, $g_{ij}\equiv g_{ij}^{(3)} \;(i<j)$ and $a_{ijk}\equiv a_{ijk}^{(3)}\;(i<j<k)$ for (respectively) the object, 1-morphism and 2-morphism of $\G^3$ corresponding to the vertex $(i)$, edge $(ij)$ and face $(ijk)$ of $\Delta^2$; thus $g_{ij}:x_i\rightarrow x_j$ and $a_{ijk}:g_{ik}\Rightarrow g_{jk}\circ g_{ij}$.     
\newline
\newline
With these notations in mind, one can immediately see that the twisted cocycle condition corresponds precisely to the commutativity of a tetrahedron $t$ in $\G^3$ as in \ref{2-nerve def}. Thus, such triples are in 1-1 correspondence with diagrams of simplicial sets of the form

\begin{equation}\label{vertex}
\xymatrix{\Delta^0\ar[d]_x\ar@<0.5ex>[r]\ar@<-0.5ex>[r] & \Delta^1\ar[d]_g\ar@<1ex>[r]\ar[r]\ar@<-1ex>[r] & \Delta^2\ar[d]_a\ar@<1ex>[r]\ar@<0.35ex>[r]\ar@<-0.35ex>[r]\ar@<-1ex>[r] &\Delta^3\ar[d]_t\\
			\N\G^0\ar@<0.5ex>[r]\ar@<-0.5ex>[r] & \N\G^1\ar@<1ex>[r]\ar[r]\ar@<-1ex>[r] & \N\G^2\ar@<1ex>[r]\ar@<0.4ex>[r]\ar@<-0.4ex>[r]\ar@<-1ex>[r] &\N\G^3\;\;.}
\end{equation}			
Since $\N\G^n$ is 3-coskeletal, diagrams as above are in turn the 0-simplices $$Tot_r(r\N\G^\bullet)_0=\uline{\rcsset}(r\Delta^\bullet,r\N\G^\bullet)_0=\rcsset(r\Delta^\bullet,r\N\G^\bullet).$$

\begin{defn}\label{definition of gauge}(cf. \cite[definition 1.5]{Ye1})
Let $\mb{d}=(x,g,a),\;\mb{d'}=(x',g',a')$ be a pair of descent data of $\G^\bullet$. A \emph{gauge transformation} $\mb{d}\leadsto \mb{d'}$ is a pair $(f,c)$ in which:
\begin{enumerate}
\item $f:x\rightarrow x'$ is a 1-morphism in $\G^0$ and
\item $c: d^0f\circ g_{01} \Rightarrow g'_{01}\circ d^1f$ is a 2-morphism in $\G^1$ (see diagram \ref{c})
\end{enumerate}

\begin{equation}\label{c}
 \xymatrix{x_0\ar[r]^{g_{01}}\ar@{}|{\LSwarrow c}[dr]\ar[d]_{f_0} & x_1
 \ar[d]^{f_1\quad (f_0:=d^1f,\;f_1:=d^0f)} 
 \\ x'_0\ar[r]_{g'_{01}} & x'_1} 
\end{equation}

such that the prism in $\G^2$ 
\begin{equation}\label{p}
\xymatrix@=2.5pc
{& x_1\ar@{.>}|{f_1}[ddd]^(0.72){\quad}="0"\ar|{g_{12}}[dr] &\\
 x_0\ar|{f_0}[ddd]\ar|{g_{01}}[ur]\ar|{g_{02}}[rr]\ar@{}[dddr]^(0.45){\quad}="1"
 &\ar@{=>}@/^.6pc/^{\quad c_{02}}"0";"1" {} ^{} \ar@{}[u]_(.3){\LUparrow a_{012}} & x_2\ar|{f_2}[ddd]^{\quad (c_{ij}: f_j\circ g_{ij}\Rightarrow g'_{ij}\circ f_i)}\\
 \ar@{}[r]^{\LSwarrow c_{01}} &\ar@{}[r]^{\LSwarrow c_{12}} & \\
  & x'_1\ar@{.>}|{g'_{12}}[dr] & \\ x'_0\ar@{.>}|{g'_{01}}[ur]\ar|{g'_{02}}[rr] &\ar@{}[u]_(.3){\LUparrow a'_{012}} & x'_2}
\end{equation}  
   is commutative in the sense of \ref{tetrahedron commutativity}.

\end{defn}

Let $Desc(\G^\bullet)$ denote the set of descent data of $\G^\bullet$. The relation $\mb{d} R\mb{d'}\Leftrightarrow \exists \mb{d}\leadsto\mb{d'}$ is an equivalence relation on $Desc(\G^\bullet)$ and we denote by $\oline{Desc}(\G^\bullet)$ its quotient (cf. \cite{Ye1}, definition 1.8). We now claim that:
\begin{thm}\label{translation}
For any cosimplicial 2-groupoid $\G^\bullet$, there is a (natural) isomorphism $$\oline{Desc}(\G^\bullet)\cong\pi_0Tot_r(r\N\G^\bullet)$$
\end{thm} 

In order to prove Theorem \ref{translation} we would like to view a gauge transformation $\mb{d}\leadsto \mb{d'}$ as a path between two vertices of $Tot_r(r\N\G^\bullet)$ but there is a slight problem. Given a pair of descent data, thought of as 4-tuples $\mb{d}=(x,g,a,t)$ and $\mb{d'}=(x',g',a',t')$ of the form \ref{vertex}, a path between them is an element of $$Tot_r(r\N\G^\bullet)_1=\uline{\rcsset}(r\Delta^\bullet,r\N\G^\bullet)_1=\rcsset(r\Delta^\bullet\times r\Delta^1,r\N\G^\bullet)$$ that restricts to $\mb{d}$ and $\mb{d'}$ via the maps $d^1,d^0:\xymatrix@1{\Delta^0\ar@<0.5ex>[r]\ar@<-0.5ex>[r] & \Delta^1}$. Since $\N\G^n$ is 3-coskeletal, such elements correspond to diagrams of the form
\begin{equation}\label{path}
\xymatrix{\Delta^0\times\Delta^1\ar[d]_f\ar@<0.5ex>[r]\ar@<-0.5ex>[r] & \Delta^1\times\Delta^1\ar[d]_{c'}\ar@<1ex>[r]\ar[r]\ar@<-1ex>[r] & \Delta^2\times\Delta^1\ar[d]_{p'}\\
			\N\G^0\ar@<0.5ex>[r]\ar@<-0.5ex>[r] & \N\G^1\ar@<1ex>[r]\ar[r]\ar@<-1ex>[r] & \N\G^2}
\end{equation}
that restrict to that restrict to $(x,g,a)$ and $(x,g,a)$. 

The last diagram carries an automatic `triangulation'. The map $c'$ is a diagram in $\G^1$ of the form
\begin{equation}\label{c'}
\xymatrix{x_0\ar[r]^{g_{01}}\ar|{h}[dr]^{\LSwarrow}_{\LSwarrow}\ar[d]_{f_0} & x_1\ar[d]^{f_1} \\ x'_0\ar[r]_{g'_{01}} & x'_1}
\end{equation}
which is a triangulation of \ref{c}; and similarly, the map $p'$ is a diagram in $\G^2$ which is a triangulation of \ref{p}.  

There are two possible solutions for that. The first (which was suggested by the referee) is to change the framework into crossed complexes, relying on \cite[Theorem 2.4]{BH2} and obtain a description of gauge transformations as maps of crossed complexes. The second, which we will adopt for the sake of simplicity, is to notice the following:
\begin{lem}\label{path-gauge}
Every gauge transformation $\mb{d}\leadsto \mb{d'}$ gives rise to a canonical path in $Tot_r(\N\G^\bullet)$ between $\mb{d}$ and $\mb{d'}$ and every such path gives rise to a canonical gauge transformation.
\end{lem}
\begin{proof}
Given a path between $\mb{d}$ and $\mb{d'}$, represented by a triple $(f,c',p')$ as in \ref{path}, one can compose the 2-morphisms appearing in $c'$ and in the squares of $p'$ to obtain a triple $(f,c,p)$ and hence a gauge transformation $\mb{d}\leadsto \mb{d'}$. Conversely, given a gauge transformation $(f,c):\mb{d}\leadsto \mb{d'}$, one obtains, from condition \ref{p} of definition \ref{definition of gauge} a prism $p$ in $\G^2$. Then, by inserting the 1-morphism $f_1\circ g_{01}$ as the diagonal in \ref{c} and $1_{f_1\circ g_{01}}$ in the upper triangle, one obtains a diagram of the form \ref{c'} and a similar procedure on \ref{p} yields a prism $p'$. The triple $(f,c',p')$ is the resulting path.
\end{proof}

Expressing the totalization as an end allow us to reveal its higher structure:
\begin{prop}\label{2-groupoid structure}
For a cosimplicial 2-groupoid $\G^\bullet$, there are natural isomorphisms
\begin{enumerate}
\item $Tot(\N\G^\bullet)\cong \N Tot(\G^\bullet)$;
\item $Tot_r(r\N \G^\bullet)\cong \N Tot_r(r\G^\bullet)$;
\end{enumerate}
(see definition \ref{abstract tot}).
\end{prop}
\begin{proof}
We only prove (1) as the proof of (2) is identical. Since $\N$ is a right adjoint, it commutes with limits. Relying on \cite[IX.5]{Ma}, 
\begin{equation}
\begin{split}
\N Tot(\G^\bullet)=N\left( \displaystyle\int_{[n]\in \Sim} (\G^n)^{\Delta^n}\right) \cong \displaystyle\int_{[n]\in\Sim}\N\left( (\G^n)^{\Delta^n}\right)\cong \displaystyle\int_{[n]\in\Sim}(\N\G^n)^{\Delta^n}=Tot(\N\G^\bullet)
\end{split}
\end{equation}
 where the last isomorphism comes from \cite[II, Lemma 2.9(2)]{GJ} relying on the fact that $\W$ commutes with arbitrary coproducts.
\end{proof}
Thus, we define:
\begin{defn}\label{descent 2-groupoid}
Given a cosimplicial 2-groupoid $\G^\bullet$, its \emph{descent 2-groupoid} is \\$\mb{Desc}(\G^\bullet):=Tot_r(\N\G^\bullet).$
\end{defn}
\begin{proof}[Proof of Theorem \ref{translation}]
Since $Tot_r(r\N\G^\bullet)$ is a Kan complex (being the 2-nerve of a 2-groupoid), $\pi_0Tot_r(\N\G^\bullet)=Tot_r(\N\G^\bullet)_0/\sim$ where $\mb{d}\sim \mb{d'}$ iff there is a path between them. By lemma \ref{path-gauge} this equivalence relation is equal to the gauge equivalence relation. 
\end{proof}

\end{section}
\begin{section}{Invariance of descent data}\label{Invariance of descent data}
Theorem \ref{translation} enables us to use homotopy-theoretic tools to prove invariance of descent data under weak equivalence.
We need one more simple theorem:
\begin{thm}\label{holim}
For any cosimplicial 2-groupoid $\G^\bullet$, there is a (natural) weak equivalence $Tot_r(\N\G^\bullet)\simeq \holim_{\Sim}\N\G^\bullet$, 
\end{thm}

\begin{proof}
In the simplicial model category $\rcsset_{proj}$, the homotopy limit (over $\Sim_{r}$) of a fibrant object can be described as the internal mapping space from a weakly contractible cofibrant object \cite[Theorem 19.4.6(2)]{Hi}. In our case, $\N\G^\bullet$ is fibrant and $r\Delta^\bullet$ is (weakly contractible and) cofibrant (see remark \ref{rDelta is cofibrant}). Thus, $Tot_r(r\N\G^\bullet)=\uline{\rcsset}(r\Delta^\bullet, \N\G^\bullet)\simeq \holim_{\Sim_{r}}r\N\G^\bullet$. By (\cite[Lemma 3.8]{DF}), $\holim_{\Sim_{r}}r\N\G^\bullet\sim \holim_{\Sim}\N\G^\bullet$.
\end{proof}

In light of definition \ref{descent 2-groupoid} and the previous theorem it now follows that:

\begin{cor}
A weak equivalence of cosimplicial 2-groupoids $\G^\bullet\rar \H^\bullet$ induces a weak equivalence of 2-groupoids $\mb{Desc}(\G^\bullet)\rar \mb{Desc}(\H^\bullet)$.
\end{cor}

In particular, we have:

\begin{cor}(cf. \cite[Theorem 2.4]{Ye1})
If $\G^\bullet\rar \H^\bullet$ is a weak equivalence of cosimplicial 2-groupoids, the induced map $\oline{Desc}(\G^\bullet)\rar \oline{Desc}(\H^\bullet)$ is an isomorphism of sets
\end{cor}
\begin{proof}
By Theorems \ref{translation} and \ref{holim}, the map $\oline{Desc}(\G^\bullet)\rar \oline{Desc}(\H^\bullet)$ coincides with $\pi_0(holim_{\Sim}\N\G^\bullet)\rar \pi_0(holim_{\Sim}\N\H^\bullet)$ and $\N$ and $\holim_\Sim$ preserve weak equivalences.
\end{proof}

\end{section}

\begin{section}{Descent data of cosimplicial $n$-groupoids}\label{n-descent}
The techniques of \S\ref{2-groupoids}--\S\ref{Invariance of descent data} work equaly well in higher dimensions. Here, we write down the details for the case of (strict) $n$-groupoids and the corresponding $n$-nerve $\N_{(n)}$ in the sense of \cite{St} but the same arguments work for weaker notions of $n$-groupoids, e.g. Tamsamani $n$-groupoids. We only need two ingredients. The first is that $N_{(n)}$ admits a left adjoint (and hence commutes with limits); this is true since the inclusion $n\mathcal{G}pd\hookrightarrow nCat$ admits a left adjoint $\Pi_n:nCat\rar n\mathcal{G}pd$ and thus the composite $\Pi_n\circ \tau_n$ (where $\tau_n$ is the fundamental $n$-category) is the desired left adjoint. The second ingredient is that $\N_{(n)}\G$ is a Kan complex for every $n$-groupoid $\G$; this goes back to \cite{Da}.
%and that $\N_{(n)}$ preserves weak equivalences.
%However, we do not know if the homotopy described in \S\ref{Tot as holim} extends to higher dimensions and thus cannot conclude the analogous statement. 

In light of theorem \ref{translation}, it makes sense to define:
\begin{defn}\label{n-descent data}
Let $\G^\bullet$ be a cosimplicial $n$-groupoid. Its \emph{$n$-descent data} is the simplicial set $\mb{Desc}_n(\G^\bullet):=Tot_r(\N_{(n)}\G)$.
\end{defn}
Definition \ref{n-descent data} makes sense formally, but its geometric meaning is unknown to us. 
Nevertheless, the formal reasoning of proposition \ref{2-groupoid structure} implies:
\begin{prop}
The simplicial set $\mb{Desc}_n(\G^\bullet)$ is the $n$-nerve of an $n$-groupoid.
\end{prop}

Moreover, theorem \ref{holim} generalizes immediately.
\begin{thm}
Let $\G^\bullet$ be a cosimplicial (strict) $n$-groupoid. There is a natural weak equivalence
$Tot_r(r\N_{(n)}\G^\bullet)\simeq \holim_{\Sim} \N_{(n)}\G^\bullet$. 
%In particular, if $\G^\bullet\rar \H^\bullet$ is a weak equivalence of cosimplicial n-groupoids then the map 
%$$Tot_r(r\N_{(n)}\G^\bullet)\rar Tot_r(r\N_{(n)}\H^\bullet)$$ is a weak equivalence of simplicial sets.
\end{thm}

\end{section}

\end{document}